\documentclass[12pt]{amsart}
\usepackage[latin1]{inputenc}
\usepackage{color}
\usepackage{bbm}
\usepackage{amsmath,amssymb}
\usepackage{enumerate}
\usepackage{mathrsfs}

\newtheorem{thm}{Theorem}[section]
\newtheorem{cor}[thm]{Corollary}
\newtheorem{lem}[thm]{Lemma}
\newtheorem{prop}[thm]{Proposition}
\theoremstyle{definition}

\newtheorem{rem}[thm]{Remark}

\numberwithin{equation}{section}

\makeatletter
\newcommand{\supp}{\operatorname{supp}}

\newcommand{\dif}{\,\mathrm{d}}

\newcommand{\charfun}{\ensuremath{\mathbbm 1}}

\DeclareMathOperator{\BMO}{BMO}

\allowdisplaybreaks

\begin{document}
\title{Martingale Inequalities for Spline Sequences}
\author[M. Passenbrunner]{Markus Passenbrunner}
\address{Institute of Analysis, Johannes Kepler University Linz, Austria, 4040
Linz, Altenberger Strasse 69}
\email{markus.passenbrunner@jku.at}
\begin{abstract}
	We show that D.~L\'{e}pingle's $L_1(\ell_2)$-inequality	
	\begin{equation*}
		 \Big\|  \big( \sum_n \mathbb E[f_n | \mathscr F_{n-1}]^2
		\big)^{1/2}\Big\|_1 \leq 2\cdot  \Big\| \big( \sum_n f_n^2
		\big)^{1/2} \Big\|_1, \qquad f_n\in\mathscr F_n,
	\end{equation*}
	extends to the case where we substitute the conditional expectation
	operators with orthogonal projection operators onto spline spaces and where we can
	allow that $f_n$ is contained in a suitable spline space $\mathscr
	S(\mathscr F_n)$.
	This is done provided the
	filtration $(\mathscr F_n)$ satisfies a certain regularity  condition
	depending on the degree of smoothness of the functions contained in
	$\mathscr S(\mathscr F_n)$. 
	As a by-product, we also obtain a spline version of $H_1$-$\BMO$ duality
	under this assumption.
\end{abstract}
\maketitle 

\section{Introduction}
This article is part of a series of papers that extend martingale results to
polynomial spline sequences of arbitrary order (see e.g. \cite{Shadrin2001, PassenbrunnerShadrin2014,
Passenbrunner2014, MuellerPassenbrunner2017,Passenbrunner2018, Passenbrunner2017,
KeryanPassenbrunner2017}).
In order to explain those martingale type results, we have to introduce a little bit of
terminology: Let $k$ be a positive integer, $(\mathscr F_n)$ an increasing
sequence of $\sigma$-algebras of sets in $[0,1]$ where each $\mathscr F_n$ is generated 
by a finite partition of $[0,1]$ into intervals of positive length. Moreover,
define the spline space
\[
	\mathscr S_k(\mathscr F_n) = \{f\in C^{k-2}[0,1] : f\text{ is a polynomial of order $k$ on
	each atom of $\mathscr F_n$} \}
\]
and  let $P_n^{(k)}$ be the orthogonal projection operator onto
$\mathscr S_k(\mathscr F_n)$ with respect to the
$L_2$  inner product on $[0,1]$  with the Lebesgue measure  $|\cdot|$.  The
space $\mathscr S_1(\mathscr F_n)$ consists of
piecewise constant functions and $P_n^{(1)}$ is the conditional expectation
operator with respect to the $\sigma$-algebra $\mathscr F_n$.
 Similarly to the definition of martingales, we introduce the following
notion:
let $(f_n)_{n\geq 0}$ be a sequence of integrable functions. We call this
sequence a \emph{$k$-martingale spline sequence} (adapted to $(\mathscr F_n)$)
if, for all $n$,
\[
	P_n^{(k)} f_{n+1} = f_n.
\]
For basic facts about martingales and conditional expectations, we refer to
\cite{Neveu1975}.

 Classical martingale theorems such as Doob's
	inequality or the martingale convergence theorem 
	 in fact carry over to
	$k$-martingale spline sequences corresponding to  \emph{arbitrary}
	filtrations ($\mathscr F_n$) of the above type, just by replacing
	conditional expectation operators by the projection operators
	$P_n^{(k)}$. Indeed, we have
		\begin{enumerate}[(i)]
			\item\label{it:splines1} (Shadrin's theorem) 
				there exists
				a constant $C_k$ depending only on $k$ such that 
				\[
					\sup_n\| P_n^{(k)} : L_1 \to L_1 \| \leq
					C_k,
				\]
			\item \label{it:splines2}
					(Doob's weak type inequality for
					splines) \\
				there exists a constant $C_k$  depending only on
				$k$ such that for any $k$-martingale
				spline sequence
				$(f_n)$ and any
				$\lambda>0$, 
				\begin{equation*}
					|\{ \sup_n |f_n| > \lambda \}| \leq C_k
				\frac{\sup_n\|f_n\|_{1}}{
				\lambda},
				\end{equation*}

			\item\label{it:splines3}(Doob's $L_p$ inequality for
			splines) \\
						for all $p\in (1,\infty]$ there exists a constant
				$C_{p,k}$  depending only on $p$ and
				$k$ such that for all $k$-martingale
				spline sequences
				$(f_n)$,
				\begin{equation*}
					\big\| \sup_n |f_n| \big\|_{p}
					\leq C_{p,k}
					\sup_n\|f_n\|_{p},\ 
				\end{equation*}
			\item\label{it:splines4}
				(Spline convergence theorem)\\
				if $(f_n)$ is an
				$L_1$-bounded $k$-martingale spline sequence, then
				$(f_n)$ converges
				 almost surely to some $L_1$-function,
			\item\label{it:splines5} (Spline convergence theorem,
				$L_p$-version)\\
				for $1<p<\infty$,
				if $(f_n)$ is an $L_p$-bounded $k$-martingale
				spline sequence, then $(f_n)$ converges almost
				surely and in $L_p$.
		 \end{enumerate}

	Property (i) is proved in \cite{Shadrin2001}, properties (ii) and (iii)
	in \cite{PassenbrunnerShadrin2014} and properties (iv) and (v)
	 in
	 \cite{MuellerPassenbrunner2017}, but see also \cite{Passenbrunner2018}.

Here, we continue this line of transferring martingale results to $k$-martingale
spline sequences and extend D.~L\'{e}pingle's $L_1(\ell_2)$-inequality
\cite{Lepingle1978}, which reads
\begin{equation}
	\label{eq:lepingle}
	 \Big\|  \big( \sum_n \mathbb E[f_n | \mathscr F_{n-1}]^2
	\big)^{1/2}\Big\|_1 \leq 2\cdot  \Big\| \big( \sum_n f_n^2
	\big)^{1/2} \Big\|_1,
\end{equation}
provided the sequence of (real-valued) random variables $f_n$ is adapted to the filtration
$(\mathscr F_n)$, i.e., each $f_n$ is $\mathscr F_n$-measurable. The spline
version of this inequality is contained in Theorem \ref{thm:lepingle_splines}.

This inequality is an $L_1$ extension of the following result for
$1<p<\infty$, proved by
E.~M.~Stein
\cite{Stein1970a}, that holds for \emph{arbitrary} integrable functions $f_n$:
\begin{equation}\label{eq:stein_inequality}
	\Big\| \big( \sum_n \mathbb E[f_n |\mathscr F_{n-1}]^2
	\big)^{1/2}\Big\|_p \leq a_p \Big\| \big( \sum_n f_n^2\big)^{1/2}
	\Big\|_p,
\end{equation}
for some constant $a_p$ depending only on $p$. 
This can be seen as a dual version of Doob's inequality $\| \sup_{\ell}
|\mathbb E[f | \mathscr F_\ell]| \|_p \leq c_p \|f\|_p$ for $p>1$, see
\cite{AsmarMontgomery1993}. Once we know Doob's inequality for spline
projections, which is point (iii) above, the same proof as in
\cite{AsmarMontgomery1993} works for
spline projections if we use suitable positive operators $T_n$ instead of
$P_n^{(k)}$ that also satisfy Doob's inequality and dominate the
operators $P_n^{(k)}$ pointwise (cf. Sections~\ref{subsect:T} and \ref{subsect:stein}).

The usage of those operators $T_n$ is also necessary in the extension of
inequality \eqref{eq:lepingle} to splines. D.~L\'{e}pingle's proof of \eqref{eq:lepingle} rests on an
idea by C. Herz \cite{Herz1974} of splitting $\mathbb E [f_n \cdot h_n]$ (for
$f_n$ being $\mathscr F_n$-measurable) by
Cauchy-Schwarz after introducing the square function $S_n^2 =
\sum_{\ell\leq n} f_\ell^2$:
\begin{equation}\label{eq:herz_idea}
	(\mathbb E[f_n\cdot h_n])^2 \leq \mathbb E[ f_n^2/S_n ] \cdot \mathbb
	E[S_n h_n^2]
\end{equation}
and estimating both factors on the right hand side separately. A key point in
estimating the second factor is that $S_n$ is $\mathscr F_n$-measurable, and
therefore, $\mathbb E[S_n|\mathscr F_n]=S_n$. If we
want to allow $f_n\in \mathscr S_k(\mathscr F_n)$, $S_n$ will
 not be
contained in $\mathscr S_k(\mathscr F_n)$ in general.
Under certain conditions on the filtration $(\mathscr F_n)$, we will show in
this article how to substitute $S_n$ in estimate \eqref{eq:herz_idea} by a
function $g_n\in\mathscr S_k(\mathscr F_n)$ that enjoys similar properties to
$S_n$ and allows us to proceed (cf. Section~\ref{subsect:mainduality}, in
particular Proposition~\ref{prop:choose_g} and Theorem~\ref{thm:main_duality}).
As a by-product, we obtain a spline version (Theorem \ref{thm:duality_splines}) of C.~Fefferman's theorem
\cite{Fefferman1971} on $H^1$-$\BMO$ duality.
For its martingale version, we
refer to A.~M.~Garsia's book \cite{Garsia1973} on Martingale Inequalities.

\section{Preliminaries}
In this section, we collect all tools that are needed subsequently.
\subsection{Properties of polynomials}
We will need 
Remez' inequality for polynomials:
\begin{thm} \label{thm:remez}
	Let $V\subset \mathbb R$ be a compact interval in $\mathbb R$ and $E\subset V$ a measurable
	subset. Then, for all polynomials $p$ of order $k$ (i.e. degree $k-1$) on $V$,
	\begin{equation*}
		\| p \|_{L_\infty(V)} \leq \bigg( 4 \frac{|V|}{|E|}\bigg)^{k-1} \| p
		\|_{L_\infty(E)}.
	\end{equation*}
\end{thm}
Applying this theorem with the set $E =  \{x\in V : |p(x)| \leq 8^{-k+1}\|p\|_{L_\infty(V)} \}$
immediately yields the following corollary{:}
\begin{cor} \label{cor:remez}
Let $p$ be a polynomial of order  $k$ on a compact interval $V\subset \mathbb R$. Then
\begin{equation*}
	\big|\big\{ x \in V : |p(x)| \geq 8^{-k+1} \|p\|_{L_\infty(V)} \big\}\big| \geq |V|/2.
\end{equation*}
\end{cor}

\subsection{Properties of spline functions}
For an interval $\sigma$-algebra $\mathscr F$ (i.e., $\mathscr F$ is generated
by a finite collection of intervals having positive length),
the space $\mathscr S_k(\mathscr F)$ is spanned by a very special local basis $(N_i)$, the so
called B-spline basis. It has the properties that each $N_i$ is
non-negative and each support of $N_i$ consists of at most $k$ neighboring atoms
of $\mathscr F$. Moreover, $(N_i)$ is a partition of unity, i.e., for all
$x\in[0,1]$, there exist at most $k$ functions $N_i$ so that $N_i(x)\neq 0$ and $\sum_i
N_i(x)=1$. 
In the following, we denote by $E_i$ the support
of the B-spline function $N_i$. The usual ordering of the B-splines $(N_i)$--which
we also employ here--is such that for all $i$, $\inf E_i \leq \inf
E_{i+1}$ and $\sup E_i \leq \sup E_{i+1}$.

We write $A(t)\lesssim B(t)$ to denote the existence of a constant
$C$ such that for all $t$, $A(t)\leq C B(t)$, where $t$ denote all implicit and
explicit dependencies the expression $A$ and $B$ might have. If the constant $C$
additionally depends on some parameter, we will indicate this in the text.
Similarly, the symbols $\gtrsim$ and $\simeq$ are used.

Another important property of B-splines is the following relation between
B-spline coefficients and the $L_p$-norm of the corresponding B-spline
expansions. 
\begin{thm}[B-spline stability, local and global]\label{thm:stability} 
Let $1\leq p\leq \infty$ and $g=\sum_{j} a_j N_j$. Then, for all $j$,
\begin{equation}\label{eq:lpstab}
|a_j|\lesssim |J_j|^{-1/p}\|g\|_{L_p(J_j)},
\end{equation}
where $J_j$ is an atom of $\mathscr F$ contained in $E_j$ 
having maximal length. 
Additionally,
\begin{equation}\label{eq:deboorlpstab}
\|g\|_p\simeq \| (a_j|E_j|^{1/p})\|_{\ell_p},
\end{equation}
where in both \eqref{eq:lpstab} and \eqref{eq:deboorlpstab}, the implied
constants depend only on the spline order $k$.
\end{thm}

Observe that \eqref{eq:lpstab} implies for $g\in \mathscr S_k(\mathscr F)$ and
any measurable set $A\subset [0,1]$
\begin{equation}\label{eq:stab_estimate}
	\|g\|_{L_\infty(A)} \lesssim \max_{j : |E_j\cap A|>0} \|
	g\|_{L_\infty(J_j)}.
\end{equation}

We will also need the following relation between the B-spline expansion of a
function and its expansion using B-splines of a finer grid.
\begin{thm}\label{thm:expansion}
	Let $\mathscr G\subset \mathscr F$ be two interval $\sigma$-algebras and
	denote by $(N_{\mathscr G,i})_i$ the B-spline basis of the coarser
	space $\mathscr S_k(\mathscr G)$ and by $(N_{\mathscr F,i})_i$ the
	B-spline basis 
	of the finer space $\mathscr S_k(\mathscr F)$.
	Then, given $f=\sum_{j} a_j N_{\mathscr G,j}$, we can
	expand $f$ in the basis $(N_{\mathscr F,i})_i$
	\[
		\sum_j a_j N_{\mathscr G,j} = \sum_i b_i N_{\mathscr F,i},
	\]
	where for each $i$, $b_i$ is a convex combination of the
	coefficients $a_j$ with $\supp N_{\mathscr G,j} \supseteq \supp
	N_{\mathscr F,i}$. 
\end{thm}
For those results and more information on spline functions, in
particular B-splines, we refer to \cite{Schumaker1981} or
\cite{DeVoreLorentz1993}.
We now use the B-spline basis of $\mathscr S_k(\mathscr F)$ and expand the 
orthogonal projection operator $P$ onto $\mathscr
S_k(\mathscr F)$  
in the form
\begin{equation}\label{eq:formula_P}
	Pf = \sum_{i,j} a_{ij} \Big(\int_0^1 f(x) N_i(x)\dif x\Big)\cdot
	N_j
\end{equation}
for some coefficients $(a_{ij})$. Denoting 
by $E_{ij}$ the
smallest interval containing both supports $E_i$ and $E_j$ of the B-spline
functions $N_i$ and $N_j$ respectively, we have the following
estimate for $a_{ij}$ \cite{PassenbrunnerShadrin2014}: 
there exist constants $C$ and $0<q<1$ depending only on $k$ so that for each
interval $\sigma$-algebra $\mathscr F$ and each $i,j$,
\begin{equation}\label{eq:geomdecay}
	|a_{ij}| \leq C \frac{q^{|i-j|}}{|E_{ij}|}.
\end{equation}

\subsection{Spline square functions} 
Let $(\mathscr F_n)$ be a sequence of increasing interval $\sigma$-algebras in $[0,1]$
and we assume that each $\mathscr F_{n+1}$ is generated from $\mathscr F_n$ by
the subdivision of exactly one atom of $\mathscr F_n$ into two atoms of
$\mathscr F_{n+1}$. Let $P_n$ be the orthogonal projection operator onto
$\mathscr S_k(\mathscr F_n)$. We denote  $\Delta_n f = P_n f - P_{n-1}
f$ and define the spline square function
\[
	Sf = \Big(\sum_n |\Delta_n f|^2\Big)^{1/2}.
\]
We have Burkholder's inequality for the spline square function, i.e.,
for all $1<p<\infty$ (\cite{Passenbrunner2014}), the $L_p$-norm of
the square function $Sf$ is comparable to the $L_p$-norm of $f$:
\begin{equation}\label{eq:Sf-pnorm}
	\|Sf\|_p \simeq \|f\|_p,\qquad f\in L_p
\end{equation}
with constants depending only on $p$ and $k$. Moreover, for $p=1$, it is shown
in \cite{4people2015} that
\begin{equation}\label{eq:Sf-1norm}
	\|Sf\|_1 \simeq \sup_{\varepsilon\in\{-1,1\}^{\mathbb Z}} \| \sum_n
	\varepsilon_n \Delta_n f \|_1, \qquad Sf\in L_1,
\end{equation}
with constants depending only on $k$ and where the proof of the $\lesssim$-part only 
uses Khintchine's inequality whereas the proof of the 
$\gtrsim$-part uses fine properties of the functions $\Delta_n f$.

\subsection{$L_p(\ell_q)$-spaces}\label{sec:Lplq}
For $1\leq p,q\leq \infty$, we denote by $L_p(\ell_q)$ the space of sequences of
measurable functions $(f_n)$ on $[0,1]$ so that the norm
\[
	\|(f_n)\|_{L_p(\ell_q)} = \Big(\int_0^1 \Big( \sum_{n} |f_n(t)|^q
	\Big)^{p/q}\dif t\Big)^{1/p}
\]
is finite (with the obvious modifications if $p=\infty$ or $q=\infty$).
For $1\leq p,q <\infty$, the dual space (see \cite{BenedekPanzone1961}) of 
$L_p(\ell_q)$ is $L_{p'}(\ell_{q'})$
with $p'=p/(p-1)$, $q'=q/(q-1)$ and the duality pairing
\[
	\langle (f_n), (g_n)\rangle = \int_0^1 \sum_n f_n(t) g_n(t) \dif t.
\]
H\"older's inequality takes the form  $|\langle (f_n), (g_n)\rangle| \leq \|(f_n)\|_{L_p(\ell_q)}
\|(g_n)\|_{L_{p'}(\ell_{q'})}$.

\section{Main Results}
In this section, we prove our main results. Section~\ref{subsect:T} defines and
gives properties of suitable positive operators that dominate our (non-positive)
operators $P_n$ pointwise. In Section~\ref{subsect:stein}, we use those operators to give
a spline version of Stein's inequality \eqref{eq:stein_inequality}. A useful
property of conditional expectations is the tower property $\mathbb E_{\mathscr
G} \mathbb E_{\mathscr F} f = \mathbb E_{\mathscr G} f$ for $\mathscr G\subset
\mathscr F$. In this form, it extends to the operators $(P_n)$, but not to the
operators $T$ from Section~\ref{subsect:T}. In Section~\ref{subsect:tower} we
prove a version of the tower property for those operators.
Section~\ref{subsect:mainduality} is devoted to establishing a duality estimate 
using a spline square function, which is the crucial ingredient in the proofs
of the spline versions of both L\'epingle's inequality \eqref{eq:lepingle} and
$H_1$-$\BMO$ duality in Section \ref{sec:apps}.

\subsection{The positive operators $T$}\label{subsect:T}
As above, let $\mathscr F$ be an interval $\sigma$-algebra on $[0,1]$, $(N_i)$
the B-spline basis of $\mathscr S_k(\mathscr F)$, $E_i$ the support of $N_i$ and
$E_{ij}$ the smallest
interval containing both $E_i$ and $E_j$.
Moreover, let $q$ be a positive
number smaller than $1$. Then, we define the linear operator $T = T_{\mathscr F, q, k}$
by
\[
	Tf(x) := \sum_{i,j}
	\frac{q^{|i-j|}}{|E_{ij}|} \langle f,
	\charfun_{E_i}\rangle \charfun_{E_j}(x) = \int_0^1 
	K(x,t) f(t)\dif t,
\]
where 
the kernel $K=K_{T}$ is given by
\[
	K(x,t) =  \sum_{i,j} \frac{q^{|i-j|}}{|E_{ij}|}\charfun_{E_i}(t)\cdot
	\charfun_{E_j}(x).
\]

We observe that the operator $T$ is selfadjoint (w.r.t the standard inner
product on $L_2$) and
\begin{equation}\label{eq:kernel}
	k\leq K_x := \int_0^1 K(t,x) \dif
	t\leq \frac{2(k+1)}{1-q},\qquad x\in[0,1],
\end{equation}
which, in particular, implies the boundedness of the operator $T$ on $L_1$ and
$L_\infty$:
\[
	\|Tf\|_1 \leq \frac{2(k+1)}{1-q} \|f\|_1,\qquad \|Tf\|_\infty \leq
	\frac{2(k+1)}{1-q}\|f\|_\infty.
\]
Another very important property of $T$ is that it is a positive operator, i.e. it maps non-negative functions to
non-negative functions and that $T$ satisfies Jensen's inequality in the form
\begin{equation}
	\label{eq:jensen}
	\varphi(Tf(x)) \leq K_x^{-1} T\big(\varphi( K_x
	\cdot f)\big)(x),\qquad f\in L_1, x\in [0,1],
\end{equation}
for convex functions $\varphi$. This is seen by applying the classical Jensen
inequality to the probability measure $K(t,x)\dif t/K_x$.

Let $\mathscr Mf $ denote the Hardy-Littlewood maximal function of $f\in L_1$,
i.e.,
\[
	\mathscr M f(x) = \sup_{I\ni x} \frac{1}{|I|}\int_I |f(y)|\dif y,
\]
where the supremum is taken over all subintervals of $[0,1]$ that contain the
point $x$. This operator is of weak type $(1,1)$, i.e.,
\[
	|\{ \mathscr M f > \lambda \}| \leq C \lambda^{-1} \|f\|_1, \qquad f\in
	L_1, \lambda>0
\]
for some constant $C$. Since trivially we have the estimate $\|\mathscr Mf\|_\infty \leq
\|f\|_\infty$, by Marcin\-kiewicz interpolation, for any $p>1$, there exists a
constant $C_p$ depending only on $p$
so that
\[
	\|\mathscr Mf\|_p \leq C_p \|f\|_p. 
\]
For those assertions about $\mathscr M$, we refer to (for instance)
\cite{Stein1970}.

The significance of $T$ and $\mathscr M$ at this point is that we can use
formula \eqref{eq:formula_P} and estimate 
\eqref{eq:geomdecay}  to obtain the pointwise bound
\begin{equation}\label{eq:bound}
	|Pf(x)| \leq C_1 (T|f|)(x) \leq C_2 \mathscr Mf(x),\qquad f\in L_1,
	x\in[0,1],
\end{equation}
where $T=T_{\mathscr F,q,k}$ with $q$ given by \eqref{eq:geomdecay} and
$C_1,C_2$ are two constants solely depending on $k$. In other words, the
positive operator $T$
dominates the non-positive operator $P$ pointwise, but at the same time, $T$ is
dominated by $\mathscr M$ pointwise independently of $\mathscr F$.

\subsection{Stein's inequality for splines}\label{subsect:stein}
We now use this pointwise dominating, positive operator $T$ to prove Stein's inequality
for spline projections.
For this, let $(\mathscr F_n)$ be an interval filtration on $[0,1]$ and 
$P_n$ be the orthogonal projection operator onto the space $\mathscr
S_k(\mathscr F_n)$
of splines of order $k$
corresponding to $\mathscr F_n$. 
Working with the positive operators $T_{\mathscr F_n, q, k}$ instead of the non-positive operators
$P_n$,
the proof of Stein's inequality \eqref{eq:stein_inequality} for spline
projections can be carried over from the martingale case (cf.
\cite{Stein1970a,AsmarMontgomery1993}).
For completeness, we include it here.

\begin{thm}
	Suppose that $(f_n)$ is a sequence of arbitrary integrable functions on $[0,1]$.
	Then, for $1\leq r\leq p<\infty$ or $1<p\leq r\leq \infty$,
	\begin{equation}
	\label{eq:stein}
		\| (P_n f_n) \|_{L_p(\ell_r)} \lesssim \| (f_n)
		\|_{L_p(\ell_r)}
	\end{equation}
	where the implied constant depends only on $p,r$ and $k$.
\end{thm}

\begin{proof}
	By \eqref{eq:bound}, it suffices to prove this inequality for the
	operators $T_n=T_{\mathscr F_n,q,k}$ with $q$ given by
	\eqref{eq:geomdecay} instead of the operators $P_n$.
	First observe that for $r=p=1$, the assertion follows from Shadrin's
	theorem ((i) on page \pageref{it:splines1}).
	Inequality \eqref{eq:bound} and the $L_{p'}$-boundedness of $\mathscr M$
for $1<p'\leq\infty$ 
	imply that
	\begin{equation}\label{eq:doobTn}
		\big\| \sup_{1\leq n\leq N} |T_n f| \big\|_{p'} \leq C_{p',k} \| f
		\|_{p'}, \qquad f\in L_{p'}
	\end{equation}
	with a constant $C_{p',k}$ depending on $p'$ and $k$.
	Let $1\leq p<\infty$ and $U_N : L_{p}(\ell_1^N) \to L_{p}$ be given by
	$(g_1,\ldots,g_N)\mapsto \sum_{j=1}^N T_j g_j$. Inequality
	\eqref{eq:doobTn} implies the boundedness of the adjoint $U_N^*
	: L_{p'}\to L_{p'}(\ell_\infty^N)$, $f\mapsto (T_j f)_{j=1}^N$ for
	$p'=p/(p-1)$ by a constant independent of $N$ and a fortiori
	the boundedness of $U_N$. Since $|T_j f|
	\leq T_j|f|$ by the positivity of $T_j$, letting $N\to\infty$ implies \eqref{eq:stein}
	for $T_n$ instead of $P_n$ in the case $r=1$ and outer parameter $1\leq
	p < \infty$.

	If $1<r\leq p$, we use Jensen's inequality \eqref{eq:jensen} and
	estimate \eqref{eq:kernel} to obtain
	\[
		\sum_{j=1}^N |T_j g_j|^r \lesssim \sum_{j=1}^N T_j(|g_j|^r) 
	\]
	and apply the result for $r=1$ and the outer parameter $p/r$ to get the
	result for $1\leq r\leq p<\infty$.
	The cases $1<p\leq r\leq \infty$ now just follow from this result using
	duality and the self-adjointness of $T_j$.
\end{proof}

\subsection{Tower property of $T$}\label{subsect:tower}
Next, we will prove a substitute of the tower property $\mathbb E_{\mathscr G}
\mathbb E_{\mathscr F} f=\mathbb E_{\mathscr G}f$ $(\mathscr G\subset \mathscr
F)$ for conditional expectations
that applies to the operators $T$.

To formulate this result, we need a suitable notion of regularity for
$\sigma$-algebras which we now describe.
Let $\mathscr F$ be an interval $\sigma$-algebra,  
let $(N_j)$ be the B-spline basis of $\mathscr S_k(\mathscr F)$ and denote by $E_{j}$
the support of the function $N_j$. 
The \emph{$k$-regularity parameter $\gamma_k(\mathscr F)$} is defined as 
\[
	\gamma_k(\mathscr F) := \max_i \max( |E_i| / |E_{i+1}|, |E_{i+1}| /
	|E_i| ),
\]
where the first maximum is taken over all $i$ so that $E_i$ and $E_{i+1}$ are
defined. The name $k$-regularity is motivated by the fact that each B-spline
support $E_i$ of order $k$ consists of at most $k$ (neighboring) atoms of the $\sigma$-algebra
$\mathscr F$.
\begin{prop}[Tower property of $T$]\label{prop:propertiesT}
	Let $\mathscr G\subset \mathscr F$ be two interval $\sigma$-algebras on $[0,1]$. 
	Let $S = T_{\mathscr G,\sigma,k}$ and
	 $T=T_{\mathscr F,\tau,k'}$  
			 for some $\sigma,\tau\in(0,1)$ and
			some positive integers $k,k'$.
			Then, for all $q>\max(\tau,\sigma)$,  there exists a
			constant $C$ depending on $q,k,k'$ so that 
			\begin{equation}
				\label{eq:tower}
				|ST f(x)| \leq C \cdot\gamma^k 
				\cdot(T_{\mathscr G,q,k} |f|)(x),\qquad f\in
				L_1, x\in[0,1], 
			\end{equation}
			where $\gamma = \gamma_{k}(\mathscr G)$ denotes the
			$k$-regularity parameter of $\mathscr G$.
\end{prop}
\begin{proof}
	Let $(F_i)$ be the collection of B-spline supports in $\mathscr
	S_{k'}(\mathscr
	F)$ and $(G_i)$ the
	collection of B-spline supports in $\mathscr S_{k}(\mathscr G)$. 
	Moreover, we denote by
	$F_{ij}$ the smallest interval containing $F_i$ and $F_j$ and by
	$G_{ij}$ the smallest interval containing $G_i$ and $G_j$.

	We show  \eqref{eq:tower} by showing   the following inequality for the kernels
	$K_S$ of $S$ and $K_T$ of $T$
	(cf. \eqref{eq:kernel})
	\begin{equation}\label{eq:kernelproduct}
		\int_0^1 K_{S}(x,t) K_{T}(t,s)\dif t
		\leq C \gamma^k
		\sum_{i,j} \frac{q^{|i-j|}}{|G_{ij}|}
		\charfun_{G_i}(x) \charfun_{G_j}(s),\qquad x,s\in[0,1]
	\end{equation}
	for all $q>\max(\tau,\sigma)$ and some constant $C$ depending on
	$q,k,k'$. In order to prove this inequality, we first fix $x,s\in [0,1]$ 
	and choose $i$ such that $x\in G_i$ and $\ell$ such that $s\in F_\ell$.
	Moreover, based on $\ell$, we choose $j$ so that $s\in G_j$ and $G_j \supset F_\ell$.
	There are at most $k$ choices for each of the indices $i,\ell,j$ and without
	restriction, we treat those choices separately, i.e., we only have to
	estimate the expression
	\[
	\sum_{m,r} \frac{\sigma^{|m-i|} \tau^{|r-\ell|} |G_m \cap
	F_r|}{|G_{im}| |F_{\ell r}|}.
	\]
	Since, for each $r$, there are also at most $k+k'-1$ indices $m$ so that
	$|G_m\cap F_r| >0$ (recall that $\mathscr G\subset \mathscr F$), we
	choose one such index $m=m(r)$ and estimate 
	\[
		\Sigma = \sum_r \frac{\sigma^{|m(r)-i|} \tau^{|r-\ell|} |G_{m(r)} \cap
			F_r|}{|G_{i,m(r)}| |F_{\ell r}|}.
	\]
	Now, observe that for any parameter choice of $r$ in the above sum, 
	\[
		G_{i,m(r)} \cup F_{\ell r} \supseteq (G_{ij}\setminus G_j) \cup
		G_i
	\]
	and therefore, since also $G_{m(r)}\cap F_r \subset G_{i,m(r)} \cap
	F_{\ell r}$,
	\[
		\Sigma \leq \frac{2}{|(G_{ij}\setminus G_j) \cup G_i|}\sum_r
		\sigma^{|m(r) - i|}\tau^{|r-\ell|},
	\]
	which, using the $k$-regularity parameter $\gamma= \gamma_{k}(\mathscr G)$ of
	the $\sigma$-algebra $\mathscr G$ and denoting $\lambda =
	\max(\tau,\sigma)$, we estimate by
	\begin{align*}
		\Sigma&\leq \frac{2\gamma^k}{|G_{ij}|} \sum_m \lambda^{|m-i|}\sum_{r:
		m(r)=m} \lambda^{|r-\ell|} \lesssim
		\frac{\gamma^k}{|G_{ij}|}\sum_m \lambda^{|i-m| + |m-j|} \\
		&\lesssim \frac{\gamma^k}{|G_{ij}|} \big(|i-j|+1\big)\lambda^{|i-j|},
	\end{align*}
	where the implied constants depend on $\lambda,k,k'$ and
	the estimate $\sum_{r: m(r)=m} \lambda^{|r-\ell|} \lesssim
	\lambda^{|m-j|}$ used the fact that, essentially, there are more atoms
	of $\mathscr F$ between $F_r$ and $F_\ell$ (for $r$ as in the sum) 
	than atoms of $\mathscr G$
	between $G_m$ and $G_j$. 
	Finally, we see that for any $q>\lambda$, 
	\[
		\Sigma \lesssim C\gamma^k\frac{q^{|i-j|}}{|G_{ij}|}
	\]
	for some constant $C$ depending on $q,k,k'$,
	and, as $x\in G_i$ and $s\in G_j$, this shows inequality
	\eqref{eq:kernelproduct}.
\end{proof}

As a corollary of Proposition \ref{prop:propertiesT}, we have
\begin{cor}\label{cor:tool_lepingle}
	Let $(f_n)$ be functions in $L_1$. We denote by $P_n$ the
	orthogonal projection onto $\mathscr S_{k}(\mathscr F_n)$ and by $P_n'$ the
	orthogonal projection onto $\mathscr S_{k'}(\mathscr F_n)$ for some
	positive integers $k,k'$. Moreover, let $T_n$ be
	the operator $T_{\mathscr F_n, q, k}$ from \eqref{eq:bound} dominating
	$P_n$ pointwise.

	Then, for
	any integer~$n$ and for any $1\leq p\leq \infty$,
	\[
		\Big\|\sum_{\ell\geq n} P_n \big((P_{\ell-1}' f_\ell)^2\big)
		\Big\|_p \lesssim
		\Big\|\sum_{\ell\geq n} T_n \big((P_{\ell-1}' f_\ell)^2\big)
		\Big\|_p
		\lesssim \gamma_{k}(\mathscr F_n)^k\cdot \Big\| \sum_{\ell\geq n}
		f_\ell^2\Big\|_p,
	\]
	where the implied constants only depend on $k$ and $k'$. 
\end{cor}
We remark that by Jensen's inequality and the tower property,
this is trivial for conditional
expectations $\mathbb E(\cdot | \mathscr F_n)$ instead of the operators $P_n, T_n,
P_{\ell-1}'$ even with
an absolute constant on the right hand side. 
\begin{proof}
	We denote by $T_n$ the operator $T_{\mathscr F_n, q, k}$  and by $T_n'$
	the operator $T_{\mathscr F_n, q', k'}$, where the
	parameters $q,q'<1$ are given by inequality \eqref{eq:bound} depending
	on $k$ and $k'$ respectively. Setting $U_n :=
	T_{\mathscr F_n, \max(q,q')^{1/2}, k}$,
	we perform the following chain of inequalities, 
	where we use the positivity of $T_n$ and
	\eqref{eq:bound}, Jensen's inequality for $T_{\ell-1}'$, the tower property
	for $T_n T_{\ell-1}'$ and the $L_p$-boundedness of $U_n$, respectively:
	\begin{align*}
		\Big\| \sum_{\ell\geq n} T_n\big( (P_{\ell-1}'
		f_\ell)^2\big)\Big\|_p  
		&\lesssim \Big\| \sum_{\ell \geq n} T_n \big( (T_{\ell-1}'
		|f_\ell|)^2\big)\Big\|_p \\ 
		&\lesssim \Big\| \sum_{\ell\geq n} T_n \big(T_{\ell-1}'
		f_\ell^2\big)\Big\|_p \\ 
		& \leq \| T_n(T_{n-1}' f_n^2) \|_p + 
			\Big\| \sum_{\ell> n} T_n \big(T_{\ell-1}'
			f_\ell^2\big)\Big\|_p \\
		&\lesssim \| f_n^2 \|_p + \gamma_{k}(\mathscr F_n)^k \cdot 
			\Big\| \sum_{\ell> n} 
			U_n(f_\ell^2)\Big\|_p	\\
		&\lesssim \gamma_{k}(\mathscr F_n)^k \cdot
		\Big\| \sum_{\ell\geq n} f_\ell^2 \Big\|_p,
	\end{align*}
	where the implied constants only depend on $k$ and $k'$.
\end{proof}

\subsection{A duality estimate using a spline square function}\label{subsect:mainduality}
In order to give the desired duality estimate contained in Theorem
\ref{thm:main_duality}, we need the following construction of a function
$g_n\in\mathscr S_k(\mathscr F_n)$ based on a spline square function.

\begin{prop}\label{prop:choose_g}
	Let $(f_n)$ be a sequence of functions with
	$f_n\in \mathscr S_k(\mathscr F_n)$ for all $n$ and set 
	\[
		X_n:=\sum_{\ell\leq n} f_\ell^2.
	\]
	Then, there exists a sequence of non-negative functions
	$g_n\in \mathscr S_{k}(\mathscr F_n)$ so that for each~$n$,
	\begin{enumerate}
		\item $g_n \leq g_{n+1}$,
		\item $X_n^{1/2} \leq  g_n$  
		\item $\mathbb E g_n  \lesssim \mathbb E X_n^{1/2}$, 
			where the implied constant depends on $k$ and on
			$\sup_{m\leq n}\gamma_{k}(\mathscr F_m)$.
	\end{enumerate}
\end{prop}

For the proof of this result, we need the following simple lemma.
\begin{lem}\label{lem:phi}
	Let $c_1$ be a positive constant and
	let $(A_j)_{j=1}^N$ be a sequence of atoms in $\mathscr F_n$.
	Moreover, let
	$\ell : \{1,\ldots, N\} \to \{1,\ldots, n\}$ and, for each $j\in
	\{1,\ldots, N\}$, let $B_j$ be a subset of an atom $D_j$ of $\mathscr
	F_{\ell(j)}$ with
	\begin{equation}
		\label{eq:ass_lemma}
		|B_j| \geq c_1 \sum_{\substack{i:\ell(i)\geq\ell(j), \\
			D_i\subseteq D_j}}
		|A_i |.
	\end{equation}

	Then, there exists a set-valued mapping $\varphi$ on $\{1,\ldots, N\}$
	so that 
	\begin{enumerate}
		\item $|\varphi(j)| = c_1 |A_j|$ for all $j$,
		\item $\varphi(j) \subseteq B_j$ for all $j$,
		\item $\varphi(i) \cap \varphi(j) = \emptyset$ for all $i\neq
			j$.
	\end{enumerate}
\end{lem}
\begin{proof}
	Without restriction, we assume that the sequence $(A_j)$ is enumerated
	such that $\ell(j+1) \leq \ell(j)$ for all $1\leq j\leq N-1$. We first
	choose $\varphi(1)$ as an arbitrary (measurable) subset of $B_1$ with
	measure $c_1|A_1|$, which is possible by assumption
	\eqref{eq:ass_lemma}. 
	Next, we assume that for $1\leq j\leq j_0$, we have constructed
	$\varphi(j)$ with the properties
	\begin{enumerate}
		\item $|\varphi(j)| = c_1|A_j|$, 
		\item $\varphi(j)\subseteq B_j$, 
		\item $\varphi(j) \cap \cup_{i<j} \varphi(i) = \emptyset$.
	\end{enumerate}
	Based on that, we now construct $\varphi(j_0 + 1)$. Define the index
	sets
	$\Gamma = \{ i : \ell(i) \geq \ell(j_0 +1), D_i \subseteq D_{j_0+1} \}$
	and $\Lambda = \{ i : i\leq j_0+1, D_i\subseteq D_{j_0+1} \}$. Since we
	assumed that $\ell$ is decreasing, we have $\Lambda \subseteq \Gamma$
	and by the nestedness of the $\sigma$-algebras $\mathscr F_n$, we have 
	for $i\leq j_0+1$ that either $D_i\subset D_{j_0+1}$ or $|D_i \cap
	D_{j_0+1}|=0$. This implies
	\begin{align*}
		\Big| B_{j_0+1}\setminus \bigcup_{i\leq j_0} \varphi(i) \Big|
		&= |B_{j_0 + 1}| - \Big| B_{j_0+1}\cap \bigcup_{i\leq j_0}
		\varphi(i) \Big| \\
		&\geq c_1 \sum_{i\in \Gamma} |A_i| - \Big| D_{j_0+1}
		\cap\bigcup_{i\leq j_0}
		\varphi(i) \Big| \\
		&\geq c_1 \sum_{i\in \Lambda} |A_i| - \Big| 
		\bigcup_{i\in\Lambda\setminus\{j_0+1\}}
		\varphi(i) \Big| \\
		&\geq c_1\sum_{i\in \Lambda} |A_i| -
		\sum_{i\in\Lambda\setminus \{j_0+1\}} c_1 |A_i| = c_1
		|A_{j_0+1}|.
	\end{align*}
	Therefore, we can choose $\varphi(j_0 +1) \subseteq B_{j_0+1}$ that is
	disjoint to $\varphi(i)$ for any $i\leq j_0$	and $|\varphi(j_0+1)| =
	c_1 |A_{j_0+1}|$ which completes the proof.
\end{proof}

\begin{proof}[Proof of Proposition \ref{prop:choose_g}]
	Fix $n$ and let $(N_{n,j})$ be the B-spline basis of $\mathscr
	S_{k}(\mathscr F_n)$.
	Moreover, for any $j$, set $E_{n,j} = \supp N_{n,j}$ and $a_{n,j} :=
	\max_{\ell\leq n} \max_{r : E_{\ell,r} \supset E_{n,j}}\|
	X_{\ell}
	\|_{L_\infty(E_{\ell,r})}^{1/2}$  and we define $\ell(j)\leq n$ and $r(j)$
	so that  $E_{\ell(j),r(j)} \supseteq E_{n,j} $ and 
	$a_{n,j} = \| X_{\ell(j)} \|_{L_\infty(E_{\ell(j),r(j)})}^{1/2}$.
	Set	
	\[
		g_n := \sum_j a_{n,j} N_{n,j} \in \mathscr S_{k}(\mathscr F_n)
	\]
	and it will be proved subsequently that this $g_n$ has the desired properties. 

	\textsc{Property (1):}
	In order to show $g_n\leq g_{n+1}$, we use 
	Theorem \ref{thm:expansion} to write 
			\[
				g_n = \sum_j a_{n,j} N_{n,j} = \sum_j
				\beta_{n,j} N_{n+1,j},
			\]
			where $\beta_{n,j}$ is a convex combination of those
			$a_{n,r}$ with $E_{n+1,j} \subseteq E_{n,r}$, and thus
			\[
				g_n \leq \sum_j \big(\max_{r: E_{n+1,j}\subseteq
				E_{n,r}} a_{n,r}\big) N_{n+1,j}.
			\]
			By the very definition of $a_{n+1,j}$,
			we have
			\[
				\max_{r: E_{n+1,j} \subseteq E_{n,r}} a_{n,r}
				\leq a_{n+1,j},
			\]
			and therefore, $g_n\leq g_{n+1}$ pointwise, since the
			B-splines $(N_{n+1,j})_j$ are nonnegative functions.

			\textsc{Property (2): } Now we show that $X_n^{1/2}\leq g_n$. 
			Indeed, for any $x\in[0,1]$,
			\[
				g_n(x) = \sum_j a_{n,j} N_{n,j}(x) \geq \min_{j
					: E_{n,j}\ni x} a_{n,j} \geq 
					\min_{j: E_{n,j}\ni x}
					\|X_n\|_{L_\infty(E_{n,j})}^{1/2}\geq
				X_n(x)^{1/2},
			\]
			since the collection of B-splines $(N_{n,j})_j$ forms a partition of unity.

		\textsc{Property (3): } Finally, we show $\mathbb E g_n \lesssim
		\mathbb E X_n^{1/2}$, where the implied constant depends only on
		$k$ and on $\sup_{m\leq n} \gamma_k(\mathscr F_m)$.
	By B-spline stability (Theorem~\ref{thm:stability}), we estimate 
	the integral of $g_n$ by  
	\begin{equation}\label{eq:expgn}
		\mathbb E g_n \lesssim \sum_j |E_{n,j}| \cdot
		\|X_{\ell(j)}\|_{L_\infty(E_{\ell(j),r(j)})}^{1/2},
	\end{equation}
	where the implied constant only depends on $k$.
	In order to continue the estimate, we next show the inequality
	\begin{equation}\label{eq:Xl_infty_estimate}
		\| X_{\ell} \|_{L_\infty(E_{\ell,r})} \lesssim \max_{s :
			|E_{\ell,r} \cap E_{\ell,s}| >0}
			\|X_{\ell}\|_{L_\infty(J_{\ell,s})},
		\end{equation}
		where by $J_{\ell,s}$ we denote an atom of $\mathscr F_\ell$
		with $J_{\ell,s} \subset E_{\ell,s}$ of maximal length and the
		implied constant depends only on $k$.
		Indeed,
		we use Theorem~\ref{thm:stability} in the form of
	\eqref{eq:stab_estimate} to get ($f_m\in\mathscr S_k(\mathscr F_\ell)$
	for $m\leq \ell$)
	\begin{equation}\label{eq:1}
		\begin{aligned}
		\|X_\ell\|_{L_\infty(E_{\ell,r})} &\leq \sum_{m\leq \ell}
		\|f_m\|_{L_\infty(E_{\ell,r})}^2 \\
		&\lesssim \sum_{m\leq \ell} 
		\sum_{s : |E_{\ell,s} \cap E_{\ell,r}|>0}
		\|f_m\|_{L_\infty(J_{\ell,s})}^2 = \sum_{s :
			|E_{\ell,s}\cap E_{\ell,r}|>0} \sum_{m\leq \ell} \| f_m
			\|_{L_\infty(J_{\ell,s})}^2.
		\end{aligned}
	\end{equation}
	Now observe that for atoms $I$ of $\mathscr F_\ell$, due to the equivalence
	of $p$-norms of polynomials (cf. Corollary \ref{cor:remez}),
	\[
		\sum_{m\leq \ell} \|f_m\|_{L_\infty(I)}^2 \lesssim
		 \sum_{m\leq \ell}\frac{1}{|I|} \int_I f_m^2 = \frac{1}{|I|}
		\int_I X_\ell \leq \|X_\ell\|_{L_\infty(I)},
	\]
	which means that, inserting this in estimate \eqref{eq:1}, 
	\[
		\|X_\ell\|_{L_\infty(E_{\ell,r})} \lesssim \sum_{s :
			|E_{\ell,s}\cap E_{\ell,r}|>0}
		\|X_\ell\|_{L_\infty(J_{\ell,s})},
	\]
	and, since there are at most $k$ indices $s$ so that $|E_{\ell,s} \cap
	E_{\ell,r}|>0$, we have established \eqref{eq:Xl_infty_estimate}.

	We define $K_{\ell,r}$ to be an interval $J_{\ell,s}$ with
	$|E_{\ell,r}\cap E_{\ell,s}|>0$ so that 
	\[
		\max_{s :
			|E_{\ell,r} \cap E_{\ell,s}| >0}
			\|X_{\ell}\|_{L_\infty(J_{\ell,s})} =
			\|X_{\ell}\|_{L_\infty(K_{\ell,r})}.
	\]
	This means, after combining \eqref{eq:expgn} with
	\eqref{eq:Xl_infty_estimate}, we have
	\begin{equation}\label{eq:gn2}
		\mathbb E g_n \lesssim \sum_j|J_{n,j}|\cdot
		\|X_{\ell(j)}\|_{L_\infty(K_{\ell(j),r(j)})}^{1/2}.
	\end{equation}
	We now apply Lemma \ref{lem:phi} with the function $\ell$ and the choices
	\begin{align*}
		A_j &= J_{n,j}, \qquad D_j = K_{\ell(j),r(j)}, \\
		B_j &= \Big\{ t\in D_j : X_{\ell(j)}(t) \geq 8^{-2(k-1)} \|
		X_{\ell(j)} \|_{L_\infty(D_j)} \Big\}.
	\end{align*}
	In order to see assumption \eqref{eq:ass_lemma} of Lemma \ref{lem:phi},
	fix the index $j$ and
	let $i$ be such that $\ell(i)\geq \ell(j)$.
	By definition of $D_i = K_{\ell(i), r(i)}$, the smallest interval
	containing $J_{n,i}$ and $D_i$ contains at most $2k-1$ atoms of
	$\mathscr F_{\ell(i)}$ and, if $D_{i}\subset D_j$, the smallest interval
	containing $J_{n,i}$ and $D_j$ contains at most $2k-1$ atoms of
	$\mathscr F_{\ell(j)}$. This means that, in particular, $J_{n,i}$ is a
	subset of the union $V$ of $4k$ atoms of $\mathscr F_{\ell(j)}$ with
	$D_j\subset V$. Since each atom of $\mathscr F_n$ occurs at most $k$
	times in the sequence $(A_j)$, there exists a constant $c_1$ depending
	on $k$ and $\sup_{u\leq \ell(j)} \gamma_k(\mathscr F_u)\leq \sup_{u\leq
	n} \gamma_k(\mathscr F_u)$ so that 
	\[
		|D_j| \geq c_1 \sum_{\substack{i:\ell(i)\geq \ell(j) \\
		D_i\subset D_j}} |A_i|,
	\]
	which, since $|B_j|\geq |D_j|/2$ by Corollary \ref{cor:remez}, 
	shows the assumption of Lemma \ref{lem:phi} 
	and we get a set-valued function
	$\varphi$ so that $|\varphi(j)| = c_1|J_{n,j}|/2$, $\varphi(j) \subset
	B_j$, $\varphi(i)\cap \varphi(j)=\emptyset$ for all $i,j$.
	Using these properties of $\varphi$, we
	continue the estimate in \eqref{eq:gn2} and write
	\begin{align*}
		\mathbb E g_n &\lesssim \sum_j |J_{n,j}| \cdot \|X_{\ell(j)}
		\|_{L_\infty( D_j)}^{1/2} \leq 8^{k-1}\cdot\sum_{j}
		 \frac{|J_{n,j}|}{|\varphi(j)|} \int_{\varphi(j)}
		X_{\ell(j)}^{1/2} \\
		&=\frac{2}{c_1} \cdot 8^{k-1} \cdot\sum_j \int_{\varphi(j)}
		X_{\ell(j)}^{1/2} \\
		&\lesssim \sum_j \int_{\varphi(j)} X_n^{1/2} \leq
		\mathbb E X_n^{1/2},
	\end{align*}
	with constants depending only on $k$ and $\sup_{u\leq
	n}\gamma_k(\mathscr F_u)$.
\end{proof}

Employing this construction of $g_n$, we  now give 
the following duality estimate for spline projections (for the martingale case,
see for instance \cite{Garsia1973}). The martingale version of this result is the 
essential estimate in 
the proof of both L\'{e}pingle's inequality \eqref{eq:lepingle} and the 
$H^1$-$\BMO$ duality.

\begin{thm}\label{thm:main_duality}
Let $(\mathscr F_n)$ be such that
	$\gamma:=\sup_{n}
	\gamma_k(\mathscr F_n) < \infty$ and $(f_n)_{n\geq 1}$ a sequence of functions
	with
	$f_n\in \mathscr S_k(\mathscr F_n)$  for each $n$. Additionally, let $h_n\in L_1$ be
	arbitrary. Then, for any $N$,
	\[
		\sum_{n\leq N} \mathbb E[|f_n \cdot h_n|] \lesssim \mathbb
		E\Big[
		\Big(\sum_{\ell\leq N} f_\ell^2\Big)^{1/2}\Big] \cdot \sup_{n\leq N}
		\|P_n\big( \sum_{\ell=n}^N h_\ell^2 \big)\|_\infty^{1/2},
	\]
	where the implied constant depends only on $k$ and $\gamma$.
\end{thm}

\begin{proof}
Let $X_n:= \sum_{\ell\leq n} f_\ell^2$ and $f_\ell\equiv 0$ for $\ell > N$ and
$\ell\leq 0$. By Proposition
\ref{prop:choose_g}, we choose an increasing sequence $(g_n)$ of functions with $g_0=0$, $g_n\in
\mathscr S_{k}(\mathscr F_n)$ and the 
properties $X_n^{1/2}\leq g_n$ and $\mathbb E g_n \lesssim \mathbb E X_n^{1/2}$ where the
implied constant depends only on $k$ and $\gamma$.
Then, apply Cauchy-Schwarz inequality by introducing the factor $g_n^{1/2}$ to
get
\[
	\sum_n\mathbb E [ |f_n \cdot h_n|] =\sum_n \mathbb E \Big[
	\Big|\frac{f_n}{g_n^{1/2}}\cdot g_n^{1/2} h_n\Big|\Big] \leq 
	\Big[ \sum_n \mathbb E  [f_n^2/g_n] \Big]^{1/2} \cdot \Big[
	\sum_n	\mathbb E[g_n h_n^2]
	\Big]^{1/2}.
\]
We estimate each of the factors on the right hand side separately and set
\[
	\Sigma_1 :=\sum_n \mathbb E  [f_n^2/g_n], \qquad \Sigma_2:=\sum_n\mathbb
	E[g_n h_n^2].
\]
The first factor is estimated by the pointwise inequality
$X_n^{1/2}\leq g_n$:
\begin{align*}
\Sigma_1 = \mathbb E \big[ \sum_n \frac{f_n^2}{g_n}\big]  &\leq 
\mathbb E \big[ \sum_n \frac{f_n^2}{X_n^{1/2}}\big]  \\
&= \mathbb E \big[\sum_n \frac{X_n - X_{n-1}}{X_n^{1/2}} \big] \leq 2\mathbb
E \sum_n (X_n^{1/2} - X_{n-1}^{1/2}) = 2 \mathbb E X_N^{1/2}.
\end{align*}
We continue with $\Sigma_2$:
\begin{align*}
	\Sigma_2 &= \mathbb E \big[ \sum_{\ell=1}^N g_\ell h_\ell^2 \big] = \mathbb E
	\Big[ \sum_{\ell=1}^N \sum_{n=1}^\ell (g_n - g_{n-1}) h_\ell^2 \Big] \\
	&= \mathbb E \Big[ \sum_{n=1}^N (g_n - g_{n-1}) \cdot\sum_{\ell=n}^N
	h_\ell^2 \Big] \\
	&= \mathbb E \Big[ \sum_{n=1}^N P_n(g_n - g_{n-1}) \cdot\sum_{\ell=n}^N
	h_\ell^2 \Big] \\
	&=\mathbb E \Big[ \sum_{n=1}^N (g_n - g_{n-1}) \cdot P_n\big(\sum_{\ell=n}^N
	h_\ell^2\big) \Big] \\
	&\leq \mathbb E \Big[\sum_{n=1}^N (g_n - g_{n-1})\Big]
	\cdot \sup_{1\leq n\leq N} \big\| P_n \big(\sum_{\ell=n}^N h_\ell^2
	\big) \big\|_\infty,
\end{align*}
where the last inequality follows from $g_n\geq g_{n-1}$.
Noting that by the properties of $g_n$, $\mathbb E \big[\sum_{n=1}^N (g_n - g_{n-1})\big] = \mathbb E g_N \lesssim
\mathbb E X_N^{1/2}$ and combining the two parts $\Sigma_1$ and $\Sigma_2$, we
obtain the conclusion. 
\end{proof}

\section{Applications}\label{sec:apps}
We give two applications of Theorem~\ref{thm:main_duality}, (i) D.~L\'{e}pingle's inequality and 
(ii) an analogue of C.~Fefferman's $H_1$-$\BMO$ duality in the setting of splines. 
Once the results from Section 3 are known, the proofs of
the subsequent results proceed similarly to their martingale counterparts
in \cite{Lepingle1978} and \cite{Garsia1973} by using spline properties
instead of martingale properties. 

\subsection{L\'{e}pingle's inequality for splines}
\begin{thm}\label{thm:lepingle_splines}
	Let $k,k'$ be positive integers.
	Let $(\mathscr F_n)$ be an interval filtration with $\sup_n
	\gamma_k(\mathscr F_n)<\infty$ and, for any $n$, $f_n \in
	\mathscr S_{k}(\mathscr F_n)$ and $P_n'$ be the
	orthogonal projection operator on $\mathscr S_{k'}(\mathscr F_n)$. Then,
	\[
		\| (P_{n-1}' f_n) \|_{L_1(\ell_2)} = \Big\| \big( \sum_n (P_{n-1}' f_n)^2 \big)^{1/2} \Big\|_1 \lesssim
		\Big\| \big( \sum_n f_n^2\big)^{1/2} \Big\|_1 = \| (f_n)
		\|_{L_1(\ell_2)},
	\]
	where the implied constant depends only on $k$, $k'$ and $\sup_n
	\gamma_{k}(\mathscr F_n)$.
\end{thm}

\begin{proof}
We first assume that $f_n=0$ for $n>N$ 
and begin by using duality
\[
	\mathbb E \big[\big( \sum_n (P_{n-1}'f_n)^2\big)^{1/2}\big] = \sup_{(H_n)}
	\mathbb E \big[ \sum_n (P_{n-1}' f_n) \cdot
	H_n\big],
\]
where sup is taken over all $(H_n)\in L_\infty(\ell_2)$ with
$\|(H_n)\|_{L_\infty(\ell_2)} =1$.  By the
self-adjointness of $P_{n-1}'$, 
\[
	\mathbb E\big[(P_{n-1}'f_n) \cdot H_n\big] = 
	\mathbb E\big[ f_n \cdot (P_{n-1}' H_n) \big].
\]
Then we apply
Theorem \ref{thm:main_duality} for $f_n$ and $h_n=P_{n-1}'H_n$ to obtain 
(denoting by $P_n$ the orthogonal projection operator onto $\mathscr
S_k(\mathscr F_n)$)
\begin{equation}\label{eq:dual_estimate_appl}
	\sum_{n\leq N} | \mathbb E [f_n\cdot h_n] | \lesssim \mathbb E\Big[
	\big(\sum_{\ell\leq N} f_\ell^2\big)^{1/2}\Big]
	\cdot \sup_{n\leq N} \Big\| P_n \big( \sum_{\ell=n}^N (P_{\ell-1}'H_\ell)^2\big)
	\Big\|_\infty^{1/2}.
\end{equation}
To estimate the right hand side, we note that for any $n$, by Corollary
\ref{cor:tool_lepingle},
\[
	\big\| P_n \big( \sum_{\ell=n}^N (P_{\ell-1}' H_\ell)^2\big) \big\|_\infty 
	\lesssim \big\| \sum_{\ell=n}^N H_\ell^2 \big\|_\infty.
\]
Therefore, \eqref{eq:dual_estimate_appl} implies
		\[
\mathbb E \big[\big( \sum_n (P_{n-1}'f_n)^2\big)^{1/2}\big] = \sup_{(H_n)}
\mathbb E \big[ \sum_n f_n \cdot ( P_{n-1}'
	H_n)\big]
			\lesssim \mathbb E\Big[
	\big(\sum_{\ell\leq N} f_\ell^2\big)^{1/2}\Big],
		\]
		with a constant depending only on $k$,$k'$ and $\sup_{n\leq N}
		\gamma_k(\mathscr F_n)$. Letting $N$ tend to infinity, we obtain
		the conclusion.
\end{proof}

\subsection{$H_1$-$\BMO$ duality for splines}

We fix an interval filtration $(\mathscr F_n)_{n=1}^\infty$, a spline order $k$ and the
orthogonal projection operators $P_n$ onto $\mathscr S_k(\mathscr F_n)$ and
additionally, we set $P_0=0$. For $f\in L_1$, 
we introduce the notation 
\[
	\Delta_n f := P_n f - P_{n-1} f,\qquad S_n(f) := \Big( \sum_{\ell=1}^n
	(\Delta_\ell f)^2\Big)^{1/2}, \qquad S(f) = \sup_n S_n(f).
\]
Observe that for $\ell < n$ and $f,g\in L_1$,
\begin{equation}\label{eq:orthogonality}
	\mathbb E [ \Delta_\ell f \cdot \Delta_n g ] = \mathbb E [ P_\ell
	(\Delta_\ell f) \cdot \Delta_n g] = \mathbb E [ \Delta_\ell f \cdot P_\ell
	(\Delta_n g)] = 0.
\end{equation}

Let $V$ be the $L_1$-closure of $\cup_n \mathscr S_k(\mathscr F_n)$. Then, the
uniform boundedness of $P_n$ on $L_1$ implies that $P_n f\to f$ in $L_1$ for
$f\in V$. Next, set 
\[
	H_{1,k} = H_{1,k}( (\mathscr F_n) ) = \{ f\in V : \mathbb E ( S(f) ) < \infty \}
\]
and equip $H_{1,k}$ with the norm $\| f \|_{H_{1,k}} = \mathbb E
S(f)$.  
Define  
\[
	\BMO_k = \BMO_k ( (\mathscr F_n) ) = \{f\in V : \sup_n \| \sum_{\ell\geq
	n}T_n\big( (\Delta_\ell
f)^2 \big)\|_\infty < \infty\}
\]
and the corresponding quasinorm 
\[
	\| f \|_{\BMO_k} = \sup_n \big\| \sum_{\ell\geq n} T_n\big((\Delta_\ell
	f)^2 \big)\big\|_\infty^{1/2},
\]
where $T_n$ is the operator from \eqref{eq:bound} that dominates $P_n$ pointwise. 

Let us now assume $\sup_n \gamma_k(\mathscr F_n) < \infty$.
In this case we identify, similarly to $H_1$-$\BMO$-duality (cf.
\cite{Fefferman1971,Herz1974,Garsia1973}), $\BMO_k$ as the
dual space of $H_{1,k}$.

First, we use the duality estimate Theorem~\ref{thm:main_duality} and
\eqref{eq:orthogonality} to prove, for $f\in H_{1,k}$ and $h\in \BMO_k$,
\[
	\big| \mathbb E \big[ (P_n f) \cdot (P_n h) \big] \big|  \leq
	\sum_{\ell\leq n} \mathbb E \big[ |\Delta_\ell f| \cdot |\Delta_\ell
	h|\big] \lesssim \| S_n(f) \|_1 \cdot \|h\|_{\BMO_k}.
\]
This estimate also implies that the limit $\lim_n \mathbb E \big[ (P_nf)\cdot
(P_nh) \big]$ exists and satisfies
\[
\big|\lim_n \mathbb E \big[ (P_nf)\cdot
(P_nh) \big]\big| \lesssim \|f\|_{H_{1,k}} \cdot \|h\|_{\BMO_k}.
\]

On the other hand, similarly to the martingale case (see \cite{Garsia1973}), 
given a continuous linear functional $L$ on $H_{1,k}$, we extend $L$
norm-preservingly to a
continuous linear functional $\Lambda$ on $L_1(\ell_2)$, which, by Section
\ref{sec:Lplq}, has the form
\[
	\Lambda(\eta) = \mathbb E \big[ \sum_\ell \sigma_\ell \eta_\ell \big],
	\qquad \eta\in L_1(\ell_2)
\]
for some $\sigma\in L_\infty(\ell_2)$. The $k$-martingale spline sequence $h_n= \sum_{\ell\leq n}
\Delta_\ell \sigma_\ell$ is bounded in $L_2$ and therefore, by the spline
convergence theorem (\eqref{it:splines5} on page \pageref{it:splines5}), has a
limit $h\in L_2$ with $P_n h = h_n$ and which is also  contained in $\BMO_k$. Indeed, by using
Corollary~\ref{cor:tool_lepingle}, we obtain $\|h\|_{\BMO_k}
\lesssim \|\sigma\|_{L_\infty(\ell_2)} = \|\Lambda\|= \|L\|$ with a constant depending only
on $k$ and $\sup_{n}\gamma_k(\mathscr F_n)$.  Moreover, for $f\in H_{1,k}$,
since $L$ is continuous on $H_{1,k}$,
\begin{align*}
L(f) =\lim_n L(P_n f) &= \lim_n \Lambda\big((\Delta_1 f,\ldots, \Delta_n
f,0,0,\ldots) \big) \\
&=\lim_n \sum_{\ell=1}^n \mathbb E[ \sigma_\ell \cdot \Delta_\ell f ]=  \lim_n \mathbb E\big[ (P_nf)\cdot (P_nh) \big].
\end{align*}
This yields the following 
\begin{thm}\label{thm:duality_splines}
	If $\sup_n \gamma_k(\mathscr F_n)<\infty$, the linear mapping 
	\[
		u : \BMO_k \to H_{1,k}^*, \qquad h \mapsto \big( f \mapsto \lim_n
		\mathbb E\big[(P_n f)\cdot (P_n h)\big] \big)
	\]
	is bijective and satisfies
	\[
		\| u(h) \|_{H_{1,k}^*} \simeq \|h\|_{\BMO_k},
	\]
	where the implied constants depend only on $k$ and $\sup_n \gamma_k(\mathscr
	F_n)$.
\end{thm}

\begin{rem}
	We close with a few remarks concerning the above result and we assume
	that $(\mathscr F_n)$ is a sequence of increasing interval $\sigma$-algebras with 
	$\sup_n \gamma_k(\mathscr F_n) < \infty$.
	\begin{enumerate}
			\item By Khintchine's inequality, $\|Sf\|_1 \lesssim
				\sup_{\varepsilon\in \{-1,1\}^{\mathbb Z}}
				\|\sum_n\varepsilon_n\Delta_n f\|_1$. Based on
				the interval filtration $(\mathscr F_n)$, we can generate
				an interval filtration $(\mathscr G_n)$ that
				contains $(\mathscr F_n)$ as a subsequence and
				each $\mathscr G_{n+1}$ is generated from $\mathscr
				G_n$ by dividing exactly one atom of $\mathscr G_n$ into
				two atoms of $\mathscr G_{n+1}$. Denoting by
				$P_n^{\mathscr G}$ the orthogonal projection
				operator onto $\mathscr S_k(\mathscr G_n)$ and
				 $\Delta_j^{\mathscr G}= P_j^{\mathscr
				G}-P_{j-1}^{\mathscr G}$, we can write
				\[
					\sum_n\varepsilon_n\Delta_n f = \sum_n
					\varepsilon_n \sum_{j=a_n}^{a_{n+1}-1}
					\Delta_j^{\mathscr G} f
				\]	
				for some sequence $(a_n)$. By using inequalities
				\eqref{eq:Sf-1norm} and \eqref{eq:Sf-pnorm} and
				writing $(S^{\mathscr G}f)^2= \sum_j
				|\Delta_j^{\mathscr G} f|^2$,
				we obtain for $p>1$
				\[
					\|Sf\|_1 \lesssim \|S^{\mathscr G} f\|_1
					\leq \|S^{\mathscr G} f\|_p \lesssim
					\|f\|_p.
				\]
				This implies $L_p\subset H_{1,k}$
				for all $p>1$ and, by duality, $\BMO_k \subset
				L_p$ for all $p<\infty$.
			\item If $(\mathscr F_n)$ is of the form that each $\mathscr
				F_{n+1}$ is generated from $\mathscr F_n$ by
				splitting exactly one atom of $\mathscr F_n$ into two
				atoms of $\mathscr F_{n+1}$ and under the
				condition $\sup_n \gamma_{k-1}(\mathscr F_n) <
				\infty$ (which is stronger than $\sup_n
				\gamma_k(\mathscr F_n)<\infty$), 
				it is shown in \cite{4people2015} that
				\[
					\|Sf\|_{1} \simeq \|f\|_{H_1},
				\]
				where $H_1$ denotes the atomic Hardy space on
				$[0,1]$, i.e. in this case, $H_{1,k}$ coincides
				with $H_1$.

\end{enumerate}
\end{rem}

\subsection*{Acknowledgments}The author is supported by the Austrian Science Fund (FWF),
Project F5509-N26,
which is a part of the Special Research Program ``Quasi-Monte Carlo Methods: 
Theory and Applications''.

\bibliographystyle{plain}
\bibliography{steinlepingle}

\end{document}